 \documentclass[11pt]{amsart}
\usepackage{geometry}   
\usepackage{setspace}
\usepackage{graphicx}
\usepackage{amssymb, amsthm}
\usepackage{epstopdf}
\usepackage{pinlabel}
 \newtheorem{theorem}{Theorem}[section]
 \newtheorem*{maintheorem}{Main Theorem}
    \newtheorem{corollary}[theorem]{Corollary}
   \newtheorem{lemma}[theorem]{Lemma}
    \newtheorem{proposition}[theorem]{Proposition}

    \theoremstyle{definition}
\newtheorem{definition}[theorem]{Definition}

\newtheorem{example}[theorem]{Example}

  \newcommand{\Z}{\ensuremath{{\mathbb{Z}}}}
\newcommand{\R}{\ensuremath{{\mathbb{R}}}}
\newcommand{\E}{\ensuremath{{\mathbb{E}}}}
  \newcommand{\Pj}{\ensuremath{{\mathbb{P}}}}
  \newcommand{\G}{\Gamma}

 \newcommand{\pitwo}{\frac{\pi}{2}}

\newcommand{\AG}{A_\G}            

\title{ Length functions of 2-dimensional right-angled Artin groups}
\author{Ruth Charney and Max Margolis}

\thanks {R. Charney was partially supported by NSF grant DMS 0705396}

\begin{document}

\begin{abstract}  Morgan and Culler proved that a minimal action of a free group on a tree is determined by its translation length function. We prove an analogue of this theorem for two-dimensional right-angled Artin groups acting on CAT(0) rectangle complexes.
\end{abstract}

\maketitle

\section{Introduction}

Let $G$ be a finitely generated group and $G \times X \to X$ an isometric action of $G$ on a metric space $X$.  The length function of the action is the function $G \to [0, \infty)$ defined by $l(g) = \inf \{ d(x, gx)  \mid  x \in X \}$. In \cite{CuMo},  Culler and Morgan study length functions of groups acting (minimally and semi-simply) on $\R$-trees.  
They prove that such actions are determined up to equivariant isometry by their length functions.  This implies that the space of all such actions, modulo scaling, embeds in an infinite dimensional projective space $\Pj^\infty$.

In the case of a free group, $G=F_n$, this theorem has important applications.  The space of  actions of $F_n$ on a simplicial tree, up to scaling, is equivalent to the space of marked graphs introduced by Culler and Vogtmann in \cite{CuVo} in their study of automorphism groups of free groups.  This space, which we denote by  $CV_n$, is commonly known as Outer Space.     By Culler-Morgan, $CV_n$  embeds in $\Pj^\infty$.  They also prove that the image of this embedding lies in a compact subset of $\Pj^\infty$, so its closure is a compactification of $CV_n$, and points on the boundary can be described as ``very small" actions of $F_n$ on an $\R$-tree \cite{CoLu}.   This has provided an essential tool in the study of automorphism groups of free groups.

In this paper we prove a two-dimensional analogue of Culler and Morgan's theorem.  
Associated to a finite, simplicial graph $\G$ with vertex set $V$ is a group $\AG$, called a \emph{right-angled Artin group}, defined by
$$\AG = \langle V \mid \textrm{$v_iv_j=v_jv_i$ if $v_i,v_j$ are adjacent in $\G$} \rangle. $$
At the two extremes, a graph $\G$ with no edges gives rise to a free group, while a complete graph $\G$ gives a free abelian group.
In recent work with K.~Vogtmann, the first author has been studying automorphisms of right-angled Artin groups,  \cite{BCV09, CCV07, CV08, CV11}.  The (outer) automorphism groups $Out(\AG)$ interpolate between $Out(F_n)$ and $GL_n(\Z)$ and provide a context for studying the similarities and differences between these groups.  It would be particularly useful to have a good analogue of Outer Space for right-angled Artin groups, as well as a compactification of this space.  

Consider as an example, the case in which $\AG$ is a product of two free groups, $F_n \times F_m$.  This is the right-angled Artin group associated to the join of two discrete graphs.  A natural candidate for Outer Space for this group is the space whose points are actions of $\AG$ on a product of two trees.  A product of two trees is
a CAT(0) rectangle complex, that is, a piecewise Euclidean CAT(0) space whose cells are rectangles.  
More generally, every right-angled Artin group acts on a rectangle complex, namely its Cayley 2-complex.  If $\G$ has no triangles, then this complex is CAT(0) and the quotient by $\AG$ is a $K(\AG,1)$-space (see \cite{Ch07}).  In this case, we say that $\AG$ is 2-dimensional.   Thus, for 2-dimensional $\AG$ a natural analogue of Culler-Vogtmann's space would be a deformation space of actions on CAT(0) rectangle complexes.  (A different notion of Outer Space for these groups was introduced in \cite{CV08}.)

With this in mind, we are interested in an analogue of Culler and Morgan's theorem for 2-dimensional right-angled Artin groups.  In this paper we prove the following theorem.
\begin{maintheorem} Assume $\G$ has no triangles and no vertices of valence 0.  Let $X$ and $X'$ be 2-dimensional CAT(0) rectangle complexes with minimal actions of $\AG$.  If the length functions associated to the two actions are the same, then $X$ and $X'$ are equivariantly isometric.
 \end{maintheorem}
 
For free groups acting on trees, an action is said to be ``minimal" if there is no invariant subtree, or equivalently, if every edge in the tree lies in the axis of some element.  The minimality condition that appears in the theorem above is an analogue of the latter condition.  See Section \ref{minimal} for a discussion of minimality.

 The special case of the main theorem in which the complexes are required to be regular cube complexes (i.e. all edge lengths equal 1) appears in the second author's thesis \cite{Mar}.


\section{Rectangle complexes}

A cubical complex is a piecewise Euclidean complexes all of whose cells are standard Euclidean cubes $C^k=  [0,1]^k$.  
In this paper, we are interested in piecewise Euclidean complexes whose cells are rectangular, that is, each cell is isometric to a finite product of intervals $\prod [0,a_i]$, but the edge lengths may vary from cell to cell.  We assume, however, that our complex has finite shapes, that is, there are only finitely many different edge lengths.  

Most of the standard properties of cubical complexes hold more generally for rectangle complexes.   In particular, the link of  a vertex in a rectangle complex $Y$ is a
piecewise spherical simplicial complex with all edge lengths $\pitwo$, hence
$Y$ is locally CAT(0) if and only if the link of every vertex is flag  (i.e., any set of pairwise adjacent vertices in the link spans a simplex).   In the case of a 2-dimensional rectangle complex, the link is just a graph and the flag condition is equivalent to the statement  that this graph has no 3-cycles. 
We can also define \emph{walls} in $Y$, as for cubical complexes, as equivalence classes of midplanes of rectangles.  Walls are geodesic and separate $Y$ into two components.  A geodesic in $Y$ crosses each wall at most once and a geodesic which intersects a wall in a non-trivial segment must lie totally inside the wall.

\begin{definition}  Let $C_1, C_2$ be two convex subcomplexes of a CAT(0) rectangle complex $Y$.  A \emph{spanning geodesic} from $C_1$ to $C_2$ is a geodesic of minimal length, i.e., whose length is equal to the distance from $C_1$ to $C_2$.
\end{definition}

If $Y$ has finite shapes, then such a spanning geodesic always exists since the set of distances between rectangles in $Y$ is discrete. 

\begin{lemma}  \label{angles}
Let $Y$ be a CAT(0) rectangle complex with finite shapes and let $C_1, C_2$ be convex sub-complexes.
\begin{enumerate}
\item If $\alpha$ is a geodesic in $Y$ from $y_1 \in C_1$ to $y_2 \in C_2$ then $\alpha$ is a spanning geodesic if and only if the angle at $y_i$ between $\alpha$ and $C_i$ is  $\geq \frac{\pi}{2}$, for $i=1,2$.
\item There exists a spanning geodesic from $C_1$ to $C_2$ whose endpoints are vertices of $Y$.
\end{enumerate}
\end{lemma}

\begin{proof}  (1) Let $\alpha$ be as above. If the angle at one endpoint, say $y_1$, between $\alpha$ and $C_1$ is  $< \frac{\pi}{2}$, then we can ``cut off the corner" near $y_1$ to get a shorter path from $C_1$ to $C_2$, so $\alpha$ is not a spanning geodesic.
Conversely, suppose the angles at both endpoints are $\geq  \frac{\pi}{2}$.  Let $\gamma$ be a spanning geodesic with endpoints $y_3,y_4$.  Then $y_1,y_2,y_3,y_3$ span a quadrilateral in $Y$ with all angles $\geq  \frac{\pi}{2}$.  By the Flat Quadrilateral Theorem (\cite{BH} p. 181), this quadrilateral spans a Euclidean rectangle.  In particular, the opposite sides $\alpha$ and $\gamma$ of this rectangle have the same length.

(2) Let $\gamma$ be any spanning geodesic from $C_1$ to $C_2$ with endpoints $y_1,y_2$. We proceed by induction on the dimension of $Y$.  If $\dim (Y)=1$ the lemma is clear.  Suppose  $\dim (Y)>1$.  Let $\sigma$ be the smallest face of $Y$ containing the initial segment of  $\gamma$.  If $y_1$ is not a vertex, then $\gamma$ is orthogonal to $\sigma \cap C_1$ , so the initial segment of $\gamma$ is parallel to some midplane of $\sigma$.  It follows that all of $\gamma$ remains parallel to the wall $W$ containing this midplane and lies in the cellular neighborhood $N(W)$ of this wall (where $N(W)$ is the union of the rectangles with a midplane in $W$).  Note that all edges of $X$ orthogonal to $W$ must have the same length $r$ since parallel edges in a cube have the same length, so 
$N(W) = W \times [0,r]$.   
Projecting $\gamma$ orthogonally onto  $W$, gives another spanning geodesic, so we may assume that, in fact, $\gamma$ lies in $W$.  

The wall $W$ inherits the structure of a CAT(0) rectangle complex of one dimension lower than $X$.  By induction, there is a spanning geodesic $\alpha$ in $W$ from $C_1 \cap W$ to $C_2 \cap W$ which begins and ends at a vertex of $W$.  Vertices of $W$ correspond to edges of $X$ which meet $W$ orthogonally.  The projection $N(W) = W \times [0,r] \to W \times 0$ takes $\alpha$ to a path of the same length which begins and ends at a vertex of $X$.  This is the desired spanning geodesic.  
\end{proof}

\begin{lemma}\label{spanning}  Let $Y$ be as above and let $C_1, C_2, C_3$ be convex sub-complexes.  Suppose $C_1$ and  $C_2$ each intersect $C_3$ in non-empty sets, but  $C_1 \cap C_2 \cap C_3 = \emptyset$.  Then any minimal length geodesic in $C_3$ from $C_1 \cap C_3$ to $C_2 \cap C_3$ is a spanning geodesic for $C_1,C_2$. 
\end{lemma}

\begin{proof}  Let  $\alpha$ be a minimal length geodesic in $C_3$ from $C_1 \cap C_3$ to $C_2 \cap C_3$ with endpoints $p,q$.   Consider the endpoint $p$.  The minimality of $\alpha$ implies that the distance in $link(p, C_3)$  between $\alpha_p$ and $link(p, C_1\cap C_3)$ is at least $\frac{\pi}{2}$. 
Thus, $\alpha_p$ is either a vertex of $link(p, C_3) \backslash link(p, C_1)$, or it lies in an edge of
$link(p, C_3)$ neither of whose vertices are in $link(p, C_1)$.  In either case, the distance from $\alpha_p$ to $link(p, C_1)$ along {\it any} path in $link(p,X)$ is at least $ \frac{\pi}{2}$.  Thus  the angle between $\alpha$ and $C_1$  is $\geq \frac{\pi}{2}$.  The same is true at $q$, so the lemma follows from Lemma \ref{angles}.
\end{proof}


\section{Length functions and minimality}\label{minimal}

Given a group, $G$, acting by isometries on a metric space $(X, d)$, we define the \emph{translation length function} of the action to be the map $l:G \to \R$ given by
 $$l(g) = \inf \{ d(x, g.x)  | x \in X \}.$$
An important concept in the study of translation length functions is the notion of a \emph{minset}.  The minset of $g \in G$ is the set on which $l(g)$ is realized, that is, 
 $$min(g) = \{x \in X | d(x, g.x) = l(g) \}.$$ 
  The action of $G$ on $X$ is said to be \emph{semi-simple} if for all $g \in G$, $min(g) \ne \emptyset$. 
   An element of $G$ is a \emph{hyperboic} isometry, if its minset is non-empty and its translation length is non-zero.  In particular, if the action of $G$ is free and semi-simple, then every element is hyperbolic.

   More generally, for a subgroup, $H< G$, the minset of $H$ is defined by 
   $$min(H) =\bigcap_{g \in H} min(g).$$
    
   For actions on CAT(0) spaces, the structure of minsets is well understood and detailed in \cite{BH}.  
Notably, $g$ is hyperbolic if and only if there exists a geodesic line in $X$  on which $g$ acts as a non-trivial translation.  Such a line is called an \emph{axis} of g and $l(g)$ is equal to the translation length along this axis.  The minset of $g$ decomposes as a product $min(g)=Y \times \E^1$ where $Y$ is a convex subspace fixed by $g$ each line $\{y\} \times \E^1$ is an axis for $g$.  
   Moreover generally, by the Flat Torus Theorem, if $H < G$ is isomorphic to $\Z^k$, then $min(H)$ is isometric to $Y \times \E^k$ for some convex subspace $Y$, $H$ fixes $Y$ and acts on $\E^k$ by translations. 
   
Now suppose $\AG$ is a right-angled Artin group whose defining graph $\G$ has no triangles and no components consisting of a single vertex (i.e., no valence 0 vertices).  We are interested in actions of $\AG$ on 2-dimensional CAT(0)  rectangle complexes.    In this situation, the maximal rank of an abelian subgroup is 2 and for such a subgroup $H$,  the minset of a $H$ is just a single flat, $\E^2$ (since $Y$ is convex and 0-dimensional).  The rectangular structure on this flat gives an orthogonal grid and $H$ acts as a finite index subgroup of the group of translations of this grid.  In particular, $H$ contains elements which translate parallel to each of the two grid directions. We call such an element a \emph{gridline isometry}.

   As in the case of free groups acting on trees, we will need a concept of ``minimality" for our actions.  
For a free group acting semi-simply on  tree, a minimal action is defined as one for which there is no invariant subtree.  This is equivalent to requiring that every point in the tree lie in the minset of some group element, that is, every point lies on some axis.  This property is key to the proof of  Culler and Morgan's theorem.  With this in mind, we define

\begin{definition}  Assume $\G$ is has no triangles and no valence 0 vertices.  Suppose $\AG$ acts properly, semi-simply by isometries on a CAT(0) rectangle complex $X$.  We will say that the action of $\AG$ is \emph{minimal} if $X$ is 2-dimensional (so minset of $\Z^2$-subgroups are isometric to $\E^2$) and $X$ is covered by the minsets of its $\Z^2$-subgroups.
\end{definition}

We remark that the minimality condition implies that the action is cocompact and hence also that $X$ has finite shapes.  Also, since $\AG$ is torsion-free, any proper action is a free action, so every element of $\AG$ is hyperbolic.

It is easy to see that minimality implies that $X$ has no convex invariant subspace. Unlike the tree case, however, the converse is not true as the next example shows.  
\begin{example}
Let $\G$ be a pentagon with edges $e_1, \dots ,e_5$.  Construct a cubical complex $\bar X$ as follows (see Figure \ref{fig:pentagon}.)  Start with 5 disjoint tori $T_1, \dots ,T_5$ corresponding to the 5 edges of $\G$.  For each $v$ vertex of $\G$, glue a tube of length 2 between the two tori containing curves marked $v$. 
The tubes create a non-trivial loop of length 10 in the center.  Glue a pentagon made of 5 squares onto this loop.  Call the central vertex of this pentagon $x_0$.  Then it is easy to check that  the fundamental group of $\bar X$ is $\AG$ and  the links in $\bar X$ are flag.  It follows that the universal covering space $X$ of $\bar X$ is CAT(0) and has a free, cocompact action of $\AG$.  It is also easy to see that $X$ is complete and has the geodesic extension property, so by Lemma 6.20 of \cite{BH}, $X$ has no $\AG$-invariant convex subcomplex.  On the other hand, it does not satisfy our definition of minimality since the link of any vertex lying over $x_0$ is a circle of radius $\frac{5\pi}{2}$ hence cannot lie in any flat. 
\end{example}

\begin{figure}
\begin{center}
\labellist
\small\hair 2pt
\pinlabel{$x_0$} at 225 188
\endlabellist
\includegraphics[width=2.7in]{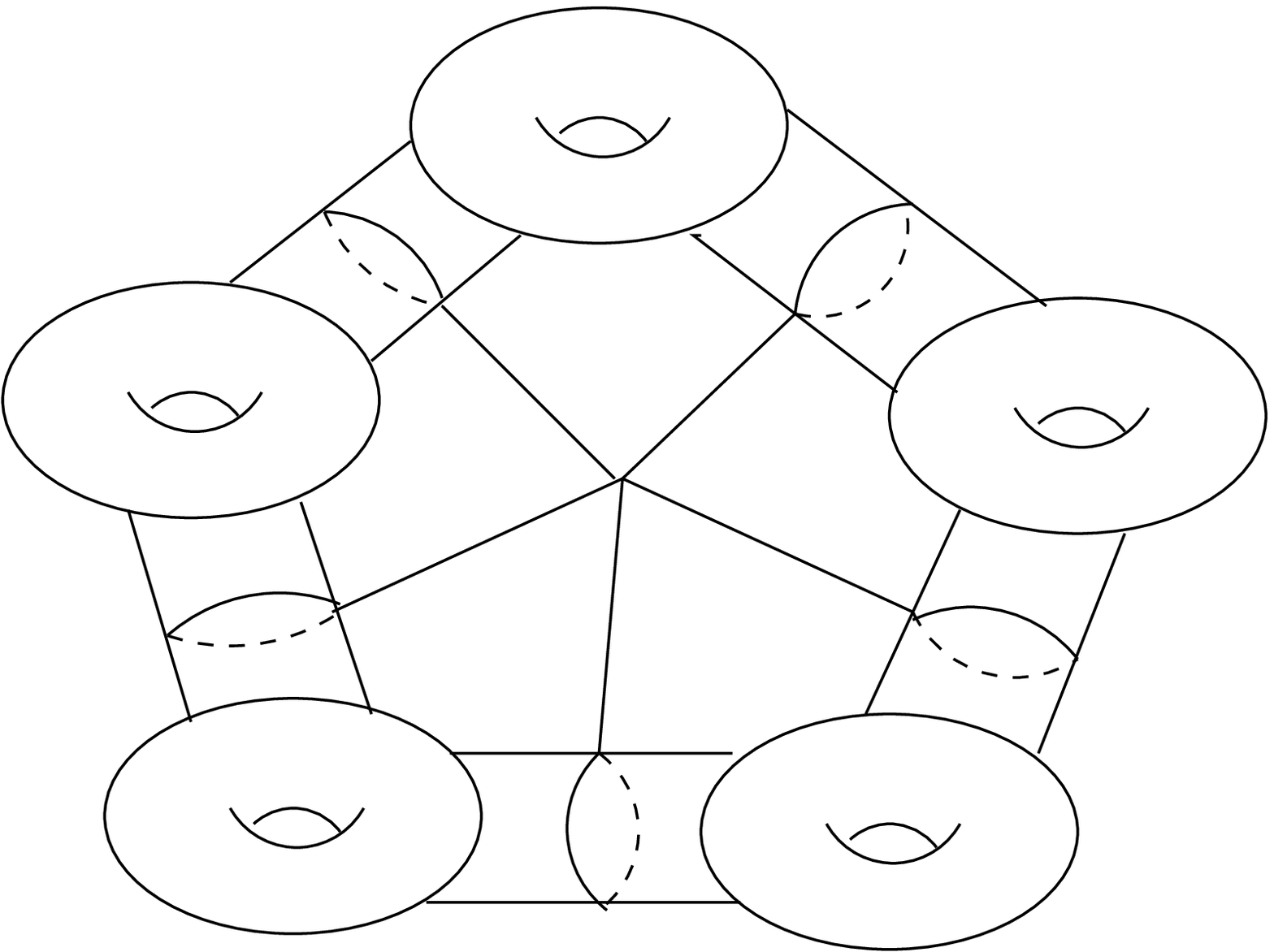}
\end{center}
\caption{} 
\label{fig:pentagon}
\end{figure}


\section{Intersections of minsets}

From now on, we assume that $\G$ has no triangles and no valence 0 vertices.  Suppose $X$ is a 2-dimensional CAT(0) rectangle complex with a minimal $\AG$ action.  Our goal is to prove that the geometry of $X$ and the action of $\AG$ is determined by the length function. In light of the definition of minimality, the
interaction between minsets of $\Z^2$-subgroups of $\AG$ will be central to the proof of the theorem.  

By definition of $\AG$, any pair of adjacent vertices in $\G$ generates a $\Z^2$-subgroup.
More generally, if $J=V_1 \ast V_2$ is a complete bipartite subgraph of $\Gamma$, then $A_J$ is a product of free groups, $A_J= F(V_1) \times F(V_2)$, so any pair of non-trivial elements $h_1 \in F(V_1)$,  $h_2 \in F(V_2)$ generates a $\Z^2$-subgroup.

\begin{definition}
	We say that a subgroup $H < A_\Gamma$ is  a \emph{basic $\Z^2$-subgroup} if $H$ is a maximal $\Z^2$-subgroup and $H$ is contained in $A_J=F(V_1) \times F(V_2)$ for some complete bipartite subgraph $J \subset \Gamma$.  
\end{definition}	

If $H$ is a basic subgroup, it is a simple exercise to show that there is a pair of generators $h_1, h_2$ for $H$, unique up to inversion, such that $h_1 \in F(V_1)$ and $h_2 \in F(V_2)$.  We refer to these as  \emph{basic generators} of $H$.

\begin{lemma}\label{lemma:BasicConj}
Every maximal  $\Z^2$-subgroup $H$ is conjugate to a basic $\Z^2$-subgroup and has a unique (up to inversion) set of generators conjugate to the basic generators.  (We call these the basic generators of $H$.)
\end{lemma}

\begin{proof}  The first assertion follows from results of Servatius \cite{Ser} which we recall here briefly.
An element of $\AG$ is called cyclically reduced if it is of minimal length in its conjugacy class. 
For any $g \in \AG$, there exists a unique cyclically reduced element conjugate to $g$.  Thus, for any $\Z^2$-subgroup $H< \AG$, we may assume up to conjugacy that $H$ contains a cyclically reduced element $h$. 

In \cite{Ser}, Servatius describes explicitly the centralizer of a cyclically reduced element.  
Write $Supp(h)$ for the set of generators appearing in a minimal length word for $h$.
It follows from Servatius' theorem that the centralizer of $h$ is cyclic unless the subgraph of $\G$ spanned by $Supp(h)$ decomposes as a non-trivial join, or the set $L$ of vertices commuting with all of $supp(h)$ is non-empty, in which case $Supp(h) \cup L$ span a join.
 Since $H$ is contained in the centralizer of $h$, this centralizer is non-cyclic and the first assertion follows.

The uniqueness of basic generators follows from the fact that the normalizer of any $\Z^2$-subgroup is equal to its centralizer.  
\end{proof}

We can now prove some simple facts about $\Z^2$-minsets. 

\begin{lemma}\label{compact}  Let $H_1, H_2$ be $\Z^2$-subgroups with minsets  $M_1$ and 
$M_2$.  Suppose $M_1 \cap M_2 \neq \emptyset$.  Then one of the following holds.
\begin{enumerate}
\item  $H_1 \cap H_2 = \{1\}$ and $M_1 \cap M_2$  is a compact (possibly degenerate) rectangle.
\item  $H_1 \cap H_2 = \langle h \rangle$, $h \neq 1$, and  $M_1 \cap M_2$  is an infinite line or infinite Euclidean strip consisting of a union of $h$-axes.
\end{enumerate}
\end{lemma}

\begin{proof}  First note that $M_1 \cap M_2$ is a convex subcomplex of $M_i$.   If $H_1 \cap H_2$ contains a non-trivial element $h$, then $M_1 \cap M_2$ is $h$-invariant, hence it is either empty or it is a union of $h$-axes.  Conversely, if $M_1 \cap M_2$ is unbounded, then it contains a ray $\alpha$ along some gridline.  Both $H_1$ and $H_2$ contain gridline isometries $h_1,h_2$ that translate along this ray.  The fact that the action is proper then implies that $h_1^n=h_2^m$ for some $n$ and $m$, so $H_1 \cap H_2$ is non-trivial.
\end{proof}

\begin{lemma}\label{branch}  Suppose $H_1, H_2$ are $\Z^2$-subgroups with 
$H_1 \cap H_2 = \langle h \rangle$, $h \neq 1$.  Then their minsets $M_1,M_2$ have branching along some $h$ axis. In particular, $h$ is a gridline isometry.
\end{lemma}

\begin{proof}  If $M_1 \cap M_2 \neq \emptyset$, this follows from the previous lemma.
If   $M_1 \cap M_2 = \emptyset$, it follows from the Flat Quadrilateral Theorem (\cite{BH}, p. 181)
that the set of spanning geodesics between them forms a convex, $h$-invariant Euclidean strip.  This strip intersects each $M_i$ in a single $h$-axis.
\end{proof}


The structure of  links at branch points in $X$ will be key to the proof of the main theorem.
We are particularly interested in branch points whose links are sufficiently complicated.  

\begin{definition} A point $x \in X$ is a \emph{corner}  if branching occurs at $x$ along a pair of orthogonal grid lines, or equivalently, if $link(x,X)$ is not isometric to the suspension of some discrete set.
\end{definition}

The terminology comes from the following observation. If $x$ is a corner as defined above, then for any minset $M$ containing $x$, we can find minsets $M_1$ and $M_2$ (not necessarily distinct) such that $x$ is a corner, in the usual sense, of the rectangle $M \cap M_1 \cap M_2$.

Certain minsets must have corners.  For example, if $v,w$ are adjacent vertices of valence at least 2 in $\G$ then it follows from Lemma \ref{branch} that the minset if $\langle v,w\rangle$ has branching along both the  $v$ and $w$ axes, hence it must have corners.
Also,  if $\G$ has more than one component, then there is some $\Z^2$-subgroup in each component whose minset has corners.  To see this, note that since $X$ is connected, some minset from each component intersects a minset in some other component.  This intersection must be compact, hence gives rise to corners.  

On the other hand, if $\G$ is the star of a single vertex $v$, then no minset in $X$ has corners since in this case. $X=min(v)= \E^1 \times T$ and all branching occurs along some axis of $v$. 
As the following lemma shows, however, in all other cases, $X$ has plenty of corners.

\begin{lemma} \label{nocorners}  Assume $\G$ is not the star of a single vertex and
suppose $H < F(V_1) \times F(V_2)$ is a basic $\Z^2$-subgroup whose minset $M$ has no corners.  Then 
\begin{enumerate}
\item either $V_1$ or $V_2$ consists of a single vertex of valence at least 2, and
\item $M$ is contained in a union of $\Z^2$-minset with corners.
\end{enumerate}
\end{lemma}

\begin{proof} We may assume without loss of generality that $V_1 \ast V_2$ is a maximal join in $\G$.
Let  $H= \langle h_1,h_2\rangle $ be a basic generating set for $H$ with $h_i \in F(V_i)$.  If $V_1$ and $V_2$ both contain at least 2 elements, then we can choose  $v_i \in V_i$ such that $v_i \neq h_i$.  Letting $H_1=\langle h_1, v_2\rangle $ and $H_2 = \langle v_1, h_2\rangle $, it follows from Lemma \ref{branch} that $M$ has branching along both $h_1$ and $h_2$ axes.  This contradicts the assumption that $M$ has no corners.  

Thus one of the sets $V_i$ consists of a single vertex, say $V_1=\{v\}$ and $V_2$ consists of all the adjacent vertices.  If $V_2$ also contains only a single vertex $w$,  then by the maximality of $F(V_1) \times F(V_2)$,  the edge between $v$ and $w$ is a component of $\G$.  Since $\G$ is not a star, it must have additional components and it follows from the discussion above that $M$ has corners.  This again contradicts our assumption, so $V_2$ must have cardinality at least 2.     

To prove the second statement of the lemma, consider the minset of $v$.   It decomposes as $\R \times T$ where the first factor is an axis for $v$ and the second factor is a tree on which the free group $F(V_2)$ acts.  For any $g \in F(V_2)$, let $\alpha(g)$ denote the axis for $g$ in $T$.  Then $min\langle v,g\rangle$  is the product $R \times \alpha(g) \subset \R \times T$.  We claim that at least one of these minset has corners.  If some $w \in V_2$ has valence at least 2, this follows from the discussion preceding the lemma.  If every  $w \in V_2$ has valence 1, then $st(v)=V_1 \ast V_2$ is an entire component of $\G$.  Since by hypothesis, $\G$ is not the star of a single vertex, it must contain other components and the claim follows from the first paragraph of the proof.  

Say $H = \langle v,h\rangle$, $h \in F(V_2)$.  Choose  $p \in F(V_2)$ such that $M'=min\langle v,p\rangle$ has corners.
Set $g=h^kp^k$ for some fixed $k$ and consider the axes for $p$, $h$, and $g$ in $T$.  It follows from basic facts about trees, that for $k$ sufficiently large, $\alpha(g) \cap \alpha(p)$ is a segment of length $> l(p)$ (i.e., its $p$-translates cover $\alpha(p)$),  and $\alpha(g) \cap \alpha(h)$ is a segment of length $> l(h)$ (i.e., its $h$-translates cover $\alpha (h)$).  From this we conclude that $min\langle v,g \rangle$ intersects $M'$ in an infinite strip containing corners,  and it intersect  $M$ is an infinite strip whose $h$-translates cover $M$.  The second statement of the lemma follows.
\end{proof}


\section{Distances between minsets}

In this section we will show that the length function determines the distances between $\Z^2$-minsets.  

\begin{theorem}\label{distance}  Let $G,H$ be maximal $\Z^2$-subgroups of $\AG$ with minsets $M_G,M_H$.
\begin{enumerate}
\item If $M_G \cap M_H \neq \emptyset$, then $ l(gh) \leq l(g) + l(h)$ for all $g \in G, h \in H$.
\item If $M_G \cap M_H= \emptyset$, then the distance $d$  between $M_G$ and $M_H$ satisfies
$$2d = \sup \{ l(gh) - l(g)-l(h) \mid g \in G, h\in H \} >0.$$
\end{enumerate}
\end{theorem}

We begin by establishing some notation.
\begin{itemize}
\item For a geodesic segment $\phi$ from $x_1$ to $x_2$ write  $\phi=[x_1,x_2]$ and $\overline\phi = [x_2, x_1]$.
\item For an element $w \in \AG$, let $\phi^w=[wx_1,wx_2]$, the translate of $\phi$ by $w$.  
\item If $\phi=[x_1,x_2]$ and $\psi=[x_2,x_3]$, denote the piecewise geodesic $\phi \cdot\psi$ by $[x_1,x_2,x_3]$. 
\item  Denote by $\phi_{x_i}$ the tangent vector to $\phi$ at $x_i$, viewed as a point in $link(x_i,X)$.
\item For two geodesic segments $\alpha=[x,y]$ and $\beta=[x,z]$, the \emph{angle between $\alpha$ and $\beta$} is the distance between $\alpha_x$ and $\beta_x$ in $link(x,X)$.   (In particular, our angles can be greater than $\pi$.)  
\end{itemize}

To motivate the proof of the theorem, let us recall what happens in the case of a free group $F$ acting on a tree $T$. If the axes for two elements $g,h \in F$ do not intersect, 
then the distance between them is exactly $\frac{1}{2}[ l(gh)-l(g)-l(h)]$.  To prove this, one notes that
the spanning geodesic $\gamma$ between the axes can be extended geodesically in either direction along the two axes.  It follows that if $x,y$ are the endpoints of $\gamma$, then the 
piecewise geodesic $[g^{-1}x,x,y,hy,hx,ghx]$ is, in fact,  geodesic (see Figure \ref{fig:tree}).  Since the last segment $[hx,ghx]$  is the $gh$-translate of the first segment $[g^{-1}x,x]$, translating this geodesic by powers of  $(gh)$, we obtain a $gh$-invariant line, namely an axis for $gh$.  The translation length along this axis is now easily seen to be $l(g)+l(h) + 2d$, where $d=$length$(\gamma)$.

\begin{figure}
\labellist
\small\hair 2pt
\pinlabel{$g^{-1}x$} at 20 203 
\pinlabel{$x$} at 120 198 
\pinlabel{$hx$} at 262 138
\pinlabel{$hgx$} at 355 138
\pinlabel{$y$} at 120 24
\pinlabel{$hy$} at 262 24
\endlabellist
\begin{center}
\includegraphics[width=3.0in]{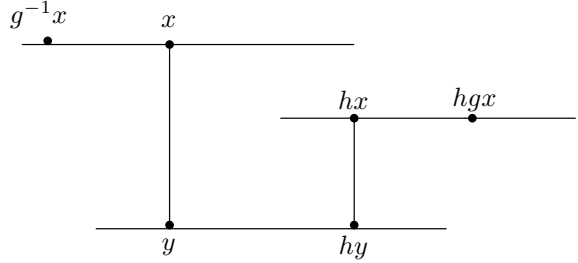}
\end{center}
\caption{Axis of $gh$ in a tree} 
\label{fig:tree}
\end{figure}

 To imitate this argument in the 2-dimensional setting,  we can again start with a spanning geodesic $\gamma$ between a pair of non-intersecting minsets $M_G$ and $M_H$ with endpoints at vertices.  If $X$ is a standard cubical complex with all side lengths equal 1, then any geodesic between vertices crosses each cube at a rational slope and in every $\Z^2$-minset, lines of rational slope are axes of some element.  In this situation, we can extend $\gamma$ geodesically, as in the tree case, along axes for some $g \in G$ and $h \in H$  to obtain an axis for $gh$ with translation length 
  $l(g)+l(h) + 2d$ (see \cite{Mar}). 
 
 In the case of a general rectangle complex, the situation is more complicated. Since the shapes of the cubes can differ as we move from one minset to another,  the directions corresponding to axes will not necessarily be rationally related. Thus, we cannot always extend $\gamma$ geodesically along the axis of some element of $G$ or $H$. 
 Nonetheless, as we will see below, we can approximate the tree picture.  Namely, for $g \in G$ and $h \in H$, consider the piecewise geodesic $[g^{-1}x,x,y,hy,hx,ghx]$. Taking translates by powers of $gh$ gives a $gh$-invariant piecewise geodesic $\phi(g,h)$, which we will call the  \emph{piecewise axis} for the pair $g,h$.  We will prove that for any $\epsilon >0$, one can choose  $g$ and $h$ such that this piecewise axis has Hausdorff distance less than  $\epsilon$ from  a true $gh$-axis (i.e., they lie within $\epsilon$-neighborhoods of each other).
  
 The proof proceeds as follows.  We begin by developing $\phi(g,h)$ onto the plane.  We then show that for appropriate choices of $g$ and $h$, we can ``straighten" this piecewise geodesic in the plane so that the resulting line corresponds to an axis for $gh$ in $X$ with the desired property.

To develop the piecewise axis onto the plane,  we need to understand the structure of the links at the endpoints of a spanning geodesic.
Let $M$ be a 2-flat in $X$, $x \in M$, and $C=lk(x,M)$.  Then $C$ is a circle of length $2\pi$ in $lk(x,X)$.
Let $t \in lk(x,X)$ be a point at distance at least $\frac{\pi}{2}$ from $C$, and let  $\Theta$ be the smallest subgraph of $lk(x,X)$ which contains $C$ and $t$.  Note that $\Theta$ need not be connected.

\begin{figure}
\labellist
\small\hair 2pt
\pinlabel{\Large$\Theta_1$} at 130 25
\pinlabel{\Large$\Theta_2$} at 475 25
\endlabellist
\begin{center}
\includegraphics[width=4.8in]{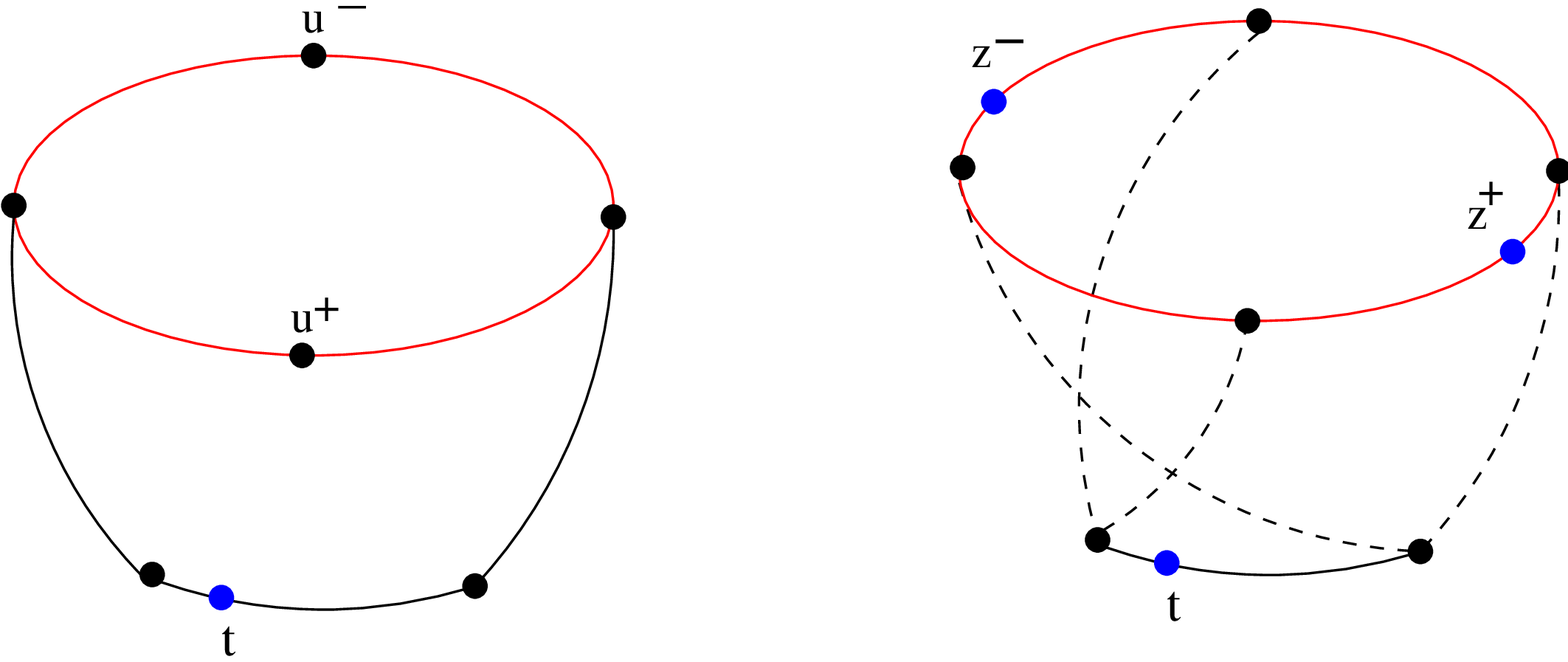}
\end{center}
\caption{Span of $t$ and $C$ in $link(x,X)$} 
\label{fig:links}
\end{figure}

\begin{lemma}\label{linktype}
  Let $\Theta_1$ and $\Theta_2$ be the graphs in Figure \ref{fig:links}.  Then 
 $\Theta$ is isomorphic to either $\Theta_1$ or to a subgraph of $\Theta_2$ (some or all of the dotted edges may be missing). 
\end{lemma}

\begin{proof}  If $t$ lies in the interior of an edge $e$ whose endpoints are connected to antipodal points in $C$ as in $\Theta_1$, then $\Theta$ must be isomorphic to $\Theta_1$. This is because no other edges can be attached to these vertices without creating a 3-cycle.  
Otherwise, either one endpoint of $e$ is not connected to $C$ by an edge, or the endpoints are connected to a pair of adjacent vertices.  In this case, $\Theta$ must be a subgraph of $\Theta_2$.
(For example, if the distance from $t$ to $C$ is greater than $\pi$, then $\Theta$ consists of just $C$ and $e$, so all of the dotted edges are missing.)

In case $t$ is itself a vertex, $\Theta$  contains only $C$ and the edges (if any) connecting $t$ to $C$.  In this case $\Theta$ is again isomorphic to a subgraph of $\Theta_2$.
\end{proof}

Suppose $\rho$ is a geodesic segment in $X$ which is transverse to every edge it crosses.  
Then any segment of $\rho$ not containing a vertex has a neighborhood isometric to a Euclidean strip.  This strip can be continued across a vertex $v$ if and only if the incoming and outgoing tangent vectors  $\rho_v^{\pm}$ lie in a circle of radius $2\pi$, i.e. a 4-cycle, in $link(v,X)$.  On the other hand, if they are not contained in a 4-cycle,  then
the tangent vectors can be perturbed toward a ``missing edge" of the 4-cycle and still remain distance $\geq \pi$ apart. In other words,  $\rho$ can be ``bent" at $v$ in at least one direction and still remain geodesic. We think of such a missing edge as  a ``slit" in the Euclidean strip about $\rho$, and we will talk about bending $\rho$ towards the slit.

Now let $\gamma$ be a spanning geodesic between two $\Z^2$-minsets  $M_G=min(G)$ and $M_H=min(H)$ with endpoints at vertices $x \in M_G$ and $y \in M_H$.    Let $\gamma_1=[x,x_1]$ (respectively $\gamma_2=[y_2,y]$) be the longest initial (respectively final) segment of $\gamma$ which has a neighborhood in $X$ isometric to a Euclidean strip.  Note that it is possible (for example if $\gamma$ lies in a single minset) that $\gamma_1=\gamma_2=\gamma$.  In this case we say $\gamma$ is \emph{unbendable}.  Otherwise, it is \emph{bendable} and we write  $\gamma=\gamma_1\gamma_0\gamma_2$.  The middle segment, $\gamma_0$, may reduce to a single point if $x_1=y_1$.

Ideally, we would like to find an isometry $g\in G$ such that at $x$, both rays of the $g$-axis through $x$ form geodesic extensions of $\gamma$, or in other words, the pair of antipodal points in $link(x,X)$ corresponding to the $g$-axis are distance at least $\pi$ from $\gamma_x$.  
Let $t=\gamma_x$ and note that since $\gamma$ is a spanning geodesic, $t$ is distance at least $\frac{\pi}{2}$ from $C=lk(x,M_G)$.  Define $\Theta$, as above, to be the smallest subgraph of $lk(x,X)$ which contains $C$ and $t$. Then by Lemma \ref{linktype}, $\Theta$ is isomorphic to either $\Theta_1$ or a subgraph of $\Theta_2$.  We say that $x$ is of \emph{type 1} or \emph{type 2} accordingly.  In the latter case, we choose an identification of $\Theta$ with a subgraph of $\Theta_2$ and call the edges of $\Theta_2$ not contained in $\Theta$ \emph{fictitious edges}. 

If $x$ is  of type 1, or if $t$ is a vertex, then there is a pair of antipodal vertices on $C$,  $u^{\pm},$ at distance at least $\pi$ from $t$.  Choosing $g$ to be a gridline isometry in direction $u^{\pm}$, gives an axis satisfying the desired condition, so if $\alpha$ is the geodesic from $x$ to $gx$, then $\overline\gamma\cdot\alpha\cdot\gamma^g =[y,x,gx,gy]$ is geodesic.

However, if $x$ is of type 2 and $t$ is not a vertex,  there are only two possible pairs of antipodal points at distance $\pi$ from $t$ in $\Theta_2$ and there is no guarantee that these pairs correspond to the axis of some $g \in G$. The best we can do is to choose $g$ with axis close these antipodal points. 

\begin{figure}
\labellist
\small\hair 2pt
\pinlabel{$\alpha^*$} at 165 100
\pinlabel{$gL$} at 135 135
\pinlabel{$gx$} at 325 135
\pinlabel{$gx_1$} at 395 135
\pinlabel{$L$} at 205 46
\pinlabel{$x_1$} at 20 47
\pinlabel{$x$} at 77 47
\endlabellist
\begin{center}
\includegraphics[width=3.5in]{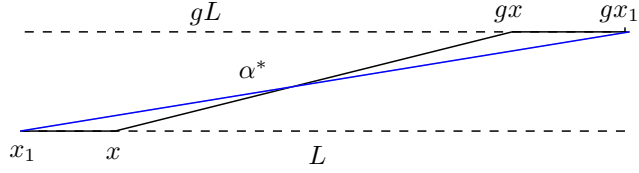}
\end{center}
\caption{$\alpha^*$ with no slits} 
\label{fig:develop}
\end{figure}

\begin{lemma}\label{alpha*} 
 For $g \in G$, let $\alpha$ denote the geodesic from $x$ to $gx$ in $M_G$.  Then for any $\epsilon >0$, there exists $g \in G$ such that the Hausdorff distance between the geodesic 
 $\alpha^*=[x_1,gx_1]$ and the piecewise geodesic  $\overline{\gamma_1}\cdot \alpha\cdot  \gamma_1^g= [x_1,x,gx,gx_1]$ is less than $\epsilon$.  Moreover, the angles
between $\alpha^*$  and this piecewise geodesic at the endpoints are also less than $\epsilon$.
\end{lemma}

\begin{proof}  As observed above, if $x$ is of type 1 of  $t = \gamma_x$  is a vertex, we can take $g$ so that $\alpha^* = \overline{\gamma_1}\cdot \alpha\cdot  \gamma_1^g$.  So assume that $x$ is of type 2 and $t$ lies in the interior of some edge of $theta$.   Choose a pair of antipodal points $z^{\pm}$  in $C$ at distance at $\pi$ from $t$.  These points are the tangent vectors to a line $L$ in $M_G$.   

By the Dirichlet's Approximation Theorem, we can find a lattice point $gx$ (i.e., a point in the $G$-orbit of $x$) arbitrarily close to $L$.  Moreover, this point can be taken to be arbitrarily far from $x$, thus the angle between $\alpha$ and $L$ at $x$ is arbitrarily small, as is the angle between $\alpha$ 
and $gL$ at $gx$, i.e. the tangent vectors to $\alpha$ lie arbitrarily close to $z^{\pm}$.  

We wish to develop the piecewise geodesic  $\overline{\gamma_1}\cdot \alpha\cdot  \gamma_1^g$ onto the plane so that the planar angles at $x$ and $gx$ are less than or equal to the actual angles in $X$.  Each of the 3 geodesic segments lies in a Euclidean strip, so the key is to determine how to assemble these strips at $x$ and $gx$.  For this, we use the identification of $\Theta$ with (a subgraph of) $\Theta_2$.  Namely,  we think of the 4-cycle  spanned by $t$ and $z^+$ as the link of $x$ in the plane, and the 4-cycle spanned by $t$ and $z^-$ as the link of $gx$ in the plane.  Note that the lines $L$ and $gL$ are parallel in $M_G$, so the segments $[x_1,x]$ and $[gx,gx_1]$ appear as parallel segments when developed onto the plane (see Figure \ref{fig:develop}).

Fictitious edges, if any, are indicated by slits (see Figure \ref{fig:slit}). Note that slits at $x$ are caused by missing edges in the 4-cycle spanned by $t$ and $z^+$ whereas slits at $gx$ are caused by
missing edges in the 4-cycle spanned by $t$ and $z^-$. Thus, it is possible to have slits at one of these points but not the other.   The planar angle measured across a slit is strictly less than the corresponding distance in the link of $x$ (or $gx$) in $X$.

The strip between $L$ and $gL$  in the plane, thus corresponds to a slit Euclidean strip in $X$ containing $\overline\gamma_1\cdot \alpha \cdot \gamma_1^g$.
If $\overline\gamma_1\cdot \alpha \cdot \gamma_1^g$ bends toward a slit, then it is already locally geodesic at that point.  If not, then it can be straightened to a geodesic $\alpha^*$ in $X$ without leaving this strip.  The lemma follows. 
\end{proof}

\begin{figure}
\labellist
\small\hair 2pt
\pinlabel{$\alpha^*$} at 280 90
\pinlabel{$gL$} at 175 135
\pinlabel{$gx$} at 325 135
\pinlabel{$gx_1$} at 395 135
\pinlabel{$L$} at 205 46
\pinlabel{$x_1$} at 20 47
\pinlabel{$x$} at 77 47
\endlabellist
\begin{center}
\includegraphics[width=3.5in]{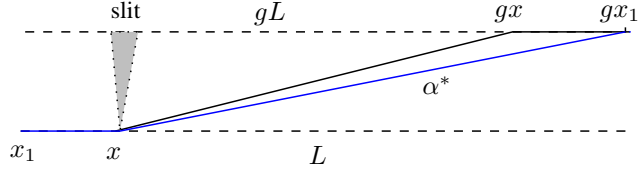}
\end{center}
\caption{$\alpha^*$ with a slit at $x$} 
\label{fig:slit}
\end{figure}

Note that we can choose $g$ so that $gx$ lies on either side of $L$. In particular, if $\gamma$ is bendable, so there are slits at $x_1$ and $gx_1$,  we can choose $g$ so that straightening 
$\overline\gamma_1\cdot \alpha \cdot \gamma_1^g$ bends it toward the slit and hence,  
$\overline\gamma_0 \cdot \alpha^* \cdot \gamma_0^g$ is still geodesic.

We can also apply the lemma at the other endpoint $y$ of $\gamma$ to find $h \in H$ with geodesic $\beta$ from $y$ to $hy$ such that the geodesic $\beta^*=[y_1,hy_1]$  is $\epsilon$-close to 
${\gamma_2}\cdot \beta\cdot  \overline\gamma_2^h=[y_1,y,hy,hy_1]$ and if $\gamma$ is bendable, then $\gamma_0\cdot \beta^* \cdot\overline\gamma_0^h$ is geodesic. Putting these together we conclude

\begin{proposition}\label{bendable}
If  $\gamma$ is bendable, then there exist $g \in G, h \in H$ such that
\begin{eqnarray}\label{eqn1}
(\beta^*)^{h^{-1}}\cdot \overline\gamma_0 \cdot \alpha^* \cdot \gamma_0^g \cdot (\beta^*)^{g}
\end{eqnarray}
 is geodesic.  Its $gh$-translates form a $gh$-axis that lies within Hausdorff distance $\epsilon$ of the piecewise axis 
\begin{eqnarray}\label{eqn2}
\phi(g,h)= \cdots \beta^{h^{-1}}\cdot \overline{\gamma}\cdot\alpha\cdot\gamma^g \cdot\beta^g  \cdots
\end{eqnarray} 
  \end{proposition}
  
  We are now ready to prove Theorem \ref{distance}.

\begin{proof}[Proof of Theorem \ref{distance}]
First assume that $M_G \cap M_H \neq \emptyset$ and let $x$ be a point in this intersection.  Then for any $g \in G$, $h \in H$, $l(gh) \leq d(x,ghx) \leq d(x,gx) +d(gx,ghx) = l(g)+l(h)$.  This proves the first statement of the theorem.

Now assume $M_G \cap M_H = \emptyset$ and let $d$ be the distance between them. Let $\gamma$ be a spanning geodesic from $M_G$ to $M_H$ whose endpoints $x,y$ are vertices.  
We first note that for any $g \in G$, $h \in H$, 
\begin{align*}
l(gh) &\leq  d(x,ghx)\\
& \leq  d(x,gx) + d(gx,gy) +d(gy,ghy) + d(ghy,ghx)\\
& =  l(g) + l(h) +2d.
\end{align*}

To prove the theorem,  it suffices to show that for any $\epsilon >0$, we can find $g, h$ such that the piecewise axis $\phi(g,h)$ lies within $\epsilon$ of some true axis for $gh$. For in this case, each segment of $\phi(g,h)$ has length at most $2\epsilon$ more than its projection on the axis, so
$$
 l(g) + l(h) +2d \leq  l(gh) + 8\epsilon.
$$

If $\gamma$ is bendable, then this follows from Proposition \ref{bendable}. If $\gamma$ is unbendable, but both endpoints of $\gamma$ are of type 1, then the piecewise axis $\phi(g,h)$
 is already geodesic, hence it is an axis for $gh$.  

So assume from now on that $\gamma$ is  unbendable and that $x$ is of type 2. 
 If $y$ is type 1, then
$\gamma\cdot\beta \cdot \overline\gamma^h$ is geodesic and has angles at $y$ and $hy$ strictly greater than $\pi$.  by Lemma \ref{alpha*}, $g$ can be chosen so that the geodesic $\alpha^*$
from $y$ to $gy$ makes arbitrarily small angles with $\overline\gamma$ and $\gamma^g$.  It follows that $ \beta^{h^{-1}}\cdot  \alpha^* \cdot \beta^g$ is geodesic and its translates form an axis for $gh$
satisfying the desired condition.

There remains the case that both links are of type 2.  
Arguing as in the proof of Lemma \ref{alpha*}, we can develop  $\overline\gamma \cdot \alpha \cdot \gamma^g$ onto the plane so that $\overline\gamma$ and $\gamma^g$ appear as parallel segments and lie in a Euclidean strip of width $\epsilon$. 
Moreover, if there is a slit at $x$, we can choose $g$ such that the bending at $x$ is towards the slit.
Since $y$ is also of type 2, we can do the same for $\gamma^{h^{-1}} \cdot \beta^{h^{-1}} \cdot \overline\gamma$ and hence likewise for $\gamma^g \cdot \beta^g \cdot \overline\gamma^{gh}$.  Continuing this process we can develop the entire piecewise axis $\phi(g,h)$ onto the plane.  Let  $\tilde\phi$ denote the image of $\phi(g,h)$ in the plane.  

The translates of $\gamma$ all appear parallel in $\tilde\phi$ and $\tilde\phi$ is invariant under the translation of the plane which takes $\gamma$ to $\gamma^{gh}$.  It follows that the convex hull of $\tilde\phi$ is a strip of width at most $2\epsilon$ bounded by a pair of parallel lines $L_1$ and $L_2$.  See Figure \ref{fig:develop2}.  This strip corresponds to a $gh$-invariant strip $E$, possibly with slits, in $X$. We will show that $E$ contains an axis for $gh$.  

If there are no slits at any vertex of $\tilde\phi$, then $L_1$ and $L_2$ lift to axes for $gh$ in $X$.  
Suppose there are slits at some vertices.  We may assume (by appropriate choice of $g$ and $h$) that  $\tilde\phi$ bends toward the slit at some vertex.  Among all such vertices, choose one, call it $v$, closest to a bounding line $L_1$ or $L_2$.  (In Figure \ref{fig:develop2}, for example, if both $gx$ and $gy$ had inward pointing slits, then we would choose $v=gx$.)   Say $v$ is closest to $L_1$ with the slit pointing downward.  Let $L_v$ be the straight line in the plane through $v$ and $ghv$.

Consider the geodesic from $v$ to $ghv$ in $X$.  It's image in the plane is the shortest path from $v$ to $ghv$ which does not cross any slit.  This path lies between $L_v$ and  $L_2$.  It follows that the $gh$-translates of this path bend (if at all) toward the slit at $v$ and translates of $v$.  Hence they lift to a $gh$-invariant geodesic in $E$, i.e., an axis. 
\end{proof}

\begin{figure}
\labellist
\small\hair 2pt
\pinlabel{$\gamma$} at 40 28
\pinlabel{$\alpha$} at 100 50
\pinlabel{$\beta$} at 310 68
\pinlabel{$hg\alpha$} at 484 80
\pinlabel{$g\gamma$} at 188 68
\pinlabel{$hg\gamma$} at 382 58
\pinlabel{$y$} at 12 5
\pinlabel{$x$} at 80 7
\pinlabel{$hgy$} at 347 27
\pinlabel{$hgx$} at 415 27
\pinlabel{$gx$} at 162 92
\pinlabel{$gy$} at 230 94
\pinlabel{$ghgx$} at 493 118
\pinlabel{$L_2$} at 205 14
\pinlabel{$L_1$} at  350 110 
\endlabellist
\begin{center}
\includegraphics[width=5.5in]{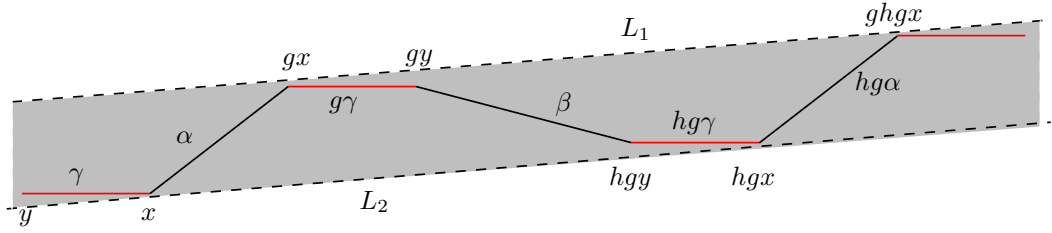}
\end{center}
\caption{The piecewise axis $\phi(g,h)$ developed onto the plane.} 
\label{fig:develop2}
\end{figure}

In the course of the proof we have shown

\begin{theorem}   Let $G,H$ be maximal $\Z^2$-subgroups with minsets $M_G,M_H$
and suppose   $M_G \cap M_H = \emptyset$.  Let $\gamma$ be a spanning geodesic from $M_G$ to $M_H$  with endpoints at vertices.   Then for any $\epsilon >0$, there exists $g \in G,  h \in H$ such that some 
axis for $gh$ intersects $\gamma$ and lies within Hausdorff distance  $\epsilon$ of the piecewise axis
$\phi(g,h)$. 
\end{theorem}

For a pair of minsets with $M_G \cap M_H \neq \emptyset$, their intersection is a (possibly degenerate) rectangle or an infinite Euclidean strip.  We will refer to all of these as ``rectangles" of side length $r_1,r_2 \in [0,\infty]$.  The following key fact is an easy consequence of the theorem.

\begin{corollary}\label{shape}  If  $M_G \cap M_H$ is a non-empty rectangle of side lengths $r_1,r_2 \in [0, \infty]$, then the dimensions $r_1,r_2$ are determined by the length function.
If either of the $r_i$ satisfy $r_i >0$,  then the length function also determines which gridline isometries in $G$ and $H$ act in the same direction along this side of the rectangle.
\end{corollary}

\begin{proof} Let $R=M_G \cap M_H$.  First note that $R$ is a single point if and only if for every $g \in G$, $l(g)=d(M_H,gM_H)$.   By Theorem \ref{distance}, the right hand side  is determined by the length function.

So assume from now on that at least one of the $r_i$ is non-zero.  In this case, we claim that the length function determines which elements $g \in G$ act as gridline isometries.  If $R$ is an infinite strip, then $G \cap H = \langle h \rangle$ and $h$ is a gridline isometry in both minsets. The orthogonal gridline isometries $g$ are then determined by the lengths of $l(g), l(h), l(g^{-1}h)$.   

Now assume $R$ is compact, and note that for any $g \in G$ with $R \cap gR = \emptyset$, 
$d(R,gR)=d(M_H,gM_H)$ is determined by the length function.  Next note that 
$$l(g) + d(R,gR) \leq d(R,g^2R)$$
with equality holding if and only if the axis for $g$ passes through two corners of $R$
(see Figure~\ref{fig:corners}), and in this case, the distance between the two corners is $l(g) - d(R,gR)$.
There are at most 4 (parallel classes of) such axes in $M_G$, two of which are along gridlines, namely the two for which 
$l(g) - d(R,gR)$  is smallest.  It follows that the length function determines which elements of $G$ act along gridlines as well as the side lengths of $R$ along these gridlines.
The same holds for $H$.  

If $r_1 \neq r_2$, this determines the identification of the rectangle $R$
 in $M_G$ with the rectangle $R$ in $M_H$ up to reflection.
 To obtain the correct orientation, consider gridline isometries $g \in G$ and $h\in H$ which translate along a side of positive length length $r= r_1$ or $r_2$.  If $h$ translates in the same direction as $g$ then 
\begin{align*} 
d(gM_H,hM_G) &=  d(gR,hR) = d(gR,R) + d(R, hR) = l(g)+l(h)-2r\\
d(gM_H,h^{-1}M_G) &= d(gR,h^{-1}R) = d(gR,R) + d(R, h^{-1}R) +r = l(g)+l(h)-r
\end{align*}
so we can distinguish between the directions $h$ and $h^{-1}$ relative to $g$.  

Finally,  if $r_1=r_2$, then the gridline isometries $h$ and $g$ act in the same direction if and only if the above two equations are satisfied and, in addition, for a gridline isometry $k \in H$ orthogonal to $h$, 
$$
d(gM_H,k^{\pm 1}M_G) \leq d(gR,k^{\pm 1}R)  \leq d(gR,R) + d(R, k^{\pm 1}R) = l(g)+l(k)-2r
$$
In particular, we must have
$d(gM_H,h^{-1}M_G) > d(gM_H,hM_G) \geq d(gM_H,k^{\pm 1}M_G)$.
This completes the proof of the second statement.
\end{proof}

\begin{figure}
\begin{center}
\includegraphics[width=2.8in]{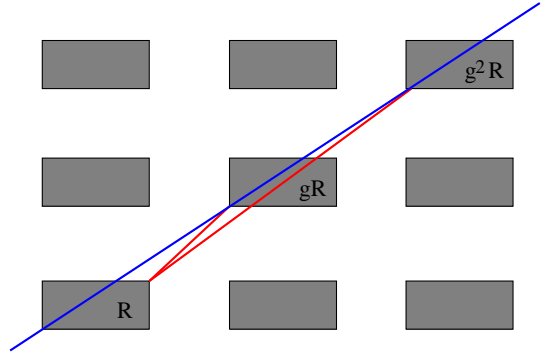}
\end{center}
\caption{The orbit of $R$ in $M_G$} 
\label{fig:corners}
\end{figure}


\section{Main theorem}

We are now ready to prove our main theorem.  

\begin{theorem}\label{main}   Assume $\G$ has no triangles and no vertices of valence 0.  Let $X$ and $X'$ be 2-dimensional CAT(0) rectangle complexes with minimal actions of $\AG$.  If the length functions associated to the two actions are the same, then $X$ and $X'$ are equivariantly isometric.
\end{theorem}

\begin{proof} 
If $\G$ is the star of a single vertex $v$, then $\AG \cong \Z \times F$ with $\Z=\langle v\rangle $ and $F$ is the free group generated the remaining vertices.  In this case, $X = min(v)= \E^1 \times T$ where the first factor is an axis for $v$ and the second factor $T$ is a tree.  It follows from the Flat Torus Theorem  (\cite{BH}, Thm II.7.1) that the action of $F$ on $X$ descends to an action $\rho$ of $F$ on $T$ such that for $(r,t) \in X$, $g \in F$,
$$ g \cdot (r,t) = (r+\lambda(g),\, \rho(g)(t))$$
where $\lambda : F \to \R$ is some homomorphism.  
It is easy to see that the minimality of $\AG$ acting on $X$ is equivalent to the minimality of $F$ acting on $T$, so by Culler and Morgan \cite{CuMo}, it suffices to show that the length function of $\AG$ acting on $X$ determines both the homomorphism $\lambda$ and the length function $l_\rho$ of $F$ acting on $T$.  For this, fix $g \in F$ and consider the action of $H=\langle v,g\rangle $ on its minset $M$.    Pick a basepoint $x_0 \in M$ and identify  $M = \E^1 \times \alpha(g)= \R^2$ with $x_0$ as the origin.  Consider the triangle  formed by $x_0, vx_0, gx_0$.  Note that the coordinates of $gx_0$ in  $\R^2$   are precisely $(\lambda(g), l_{\rho}(g))$.  This triangle is uniquely determined  by the three lengths $l(v),l(g),l(gv)$ (and the fact that $v$ translates in a positive direction along the first factor), hence so are the coordinates of $gx_0$.  

Now assume that $\G$ is not the star of a single vertex.
Fix a maximal $\Z^2$-subgroup $H$ and let $M, M'$ be the minsets of $H$ in $X$ and $X'$ respectively.  
For any other minset $M_i$ intersecting $M$, let $R_i = M \cap M_i$ and $R'_i=M' \cap M'_i$.

We first prove that if $M$ has corners, then there is a unique isometry $\phi_H : M \to M'$ such that $\phi_H$ maps $R_i$ onto $R'_i$ for all $i$.
If some $R_i$ is compact, then
by Corollary \ref{shape} and Theorem \ref{distance}, the shape of $R_i$ and the distances between its $H$-translates are determined by the length function. So there is a unique $H$-equivariant isometry  $\phi_H$ from $M$ to $M'$ that identifies $R_i$ with $R'_i$.
For any other rectangle (or strip) $R_j$, the position of   $R_j$ in $M$ is uniquely determined by its distance from the grid of rectangles formed by the $H$-translates of $R_i$.  By Lemma \ref{spanning},
the distance from $R_j$ to such a translate $hR_i$ is equal to the distance from $M_j$ to $hM_i$ so it is determined by the length function.  It follows that $\phi_H$ takes $R_j$ to $R'_j$ for all $j$.  

If $R_i$ is an infinite strip for every $M_i$, then the same is true for $R'_i$ by Lemma \ref{compact}.  Since we are assuming that $M$ has corners, there exist such strips in both gridline  directions.  By Corollary \ref{shape} and Theorem \ref{distance}, the direction (horizontal or vertical) of these strips, their widths, and the distances between them are all determined by the length function hence there is a unique $H$-equivariant isometry  $\phi_H$ from $M$ to $M'$ that identifies each $R_i $ to $R_i'$. 

We have defined an isometry $\phi_H : M \to M'$ for every $\Z^2$-minset with corners, such that all of these maps agree on overlaps. By Lemma \ref{nocorners}, these minsets cover $X$, thus the $\phi_H$ fit together to give a map $\Phi: X \to X'$. 
The uniqueness of $\phi_H$ implies that the combined map $\Phi$ is $\AG$-equivariant.  
Moreover, the induced maps $link(x, M) \to link(\phi_H(x), M')$ likewise fit together to give an isomorphism of graphs $link(x,X) \to link(\Phi(x), X')$.   It follows that $\Phi$ takes local geodesics 
to local geodesics.  Since $X$ and $X'$ are CAT(0), we conclude that $\Phi$ is a global isometry.  
This completes the proof of the theorem.
\end{proof}


\begin{thebibliography}{99}


\bibitem{BH}
{M. Bridson and A. Haefliger}, Metric Spaces of Non-positive Curvature,
 Grundlehren der Mathematischen Wissenschaften 319,  Springer-Verlag, Berlin, 1999. 
 
\bibitem{BCV09} Kai-Uwe Bux, Ruth Charney and Karen Vogtmann, \emph{Automorphisms of two-dimensional RAAGs and partially symmetric automorphisms of free groups}, Groups Geom. Dyn. \textbf{3}  (2009) no. 4, 541--554.
  
\bibitem{Ch07} Ruth Charney, \emph{An introduction to right-angled Artin groups}, Geom. Dedicata  \textbf{125} (2007) 141--158

\bibitem{CCV07}
Ruth Charney, John Crisp and Karen Vogtmann, \emph{Automorphisms of 2-dimensional right-angled Artin groups},  Geom. and Topology \textbf{11} (2007), 2227--2264.


\bibitem{CV08}
Ruth Charney and Karen Vogtmann,  
\emph{Finiteness properties of automorphism groups of right-angled Artin groups}, Bull. Lond. Math. Soc.  \textbf{41}  (2009),  no. 1, 94--102. 

\bibitem{CV11}
Ruth Charney and Karen Vogtmann,  
\emph{ Subgroups and quotients of automorphism groups of RAAGS},
to appear in Proceedings of the 2008 Georgia Topology Conference, 
Geom. and Topology Monographs.

 \bibitem{CoLu}
M.~Cohen, and M.~Lustig, Very small group actions on R-trees and Dehn twist automorphisms, \emph{Topology} {\bf 34} (3) (1995), 575--617.

 \bibitem{CuMo}
 M.~Culler and J.~Morgan, Group actions on R-trees, 
 \emph{Proc. Lond. Math. Soc.} {\bf 55} (1987), 
571--604. 
 
 
\bibitem{CuVo}
Marc Culler and Karen Vogtmann,
 \emph{Moduli of graphs and automorphisms of
  free groups}, Invent. Math. \textbf{84} (1986), no.~1, 91--119.  

 
 \bibitem{Mar}
 M.~Margolis, Length functions of right-angled Artin groups, 
 PhD thesis, Brandeis University, 2010

\bibitem{Ser}
H.~Servatius, Automorphisms of graph groups,
\emph{J. Algebra}  {\bf 126} (1989), 34--60.

\end{thebibliography}
\end{document}